\documentclass{elsart}               

\usepackage{stmaryrd}
\usepackage{graphicx}          

\usepackage{cite}
\usepackage{amsmath,amssymb,amsfonts}
\usepackage{algcompatible}
\usepackage{algorithm}
\usepackage{epstopdf}
\usepackage{hyperref}
\usepackage{textcomp}
\usepackage{bm}
\usepackage{adjustbox}
\usepackage{amssymb}
\usepackage{enumitem}
\usepackage{subcaption}
\usepackage{xcolor}

\newcommand{\norm}[1]{\left\lVert#1\right\rVert}

\newtheorem{theorem}{\bf{Theorem}}

\newtheorem{corollary}{\bf{Corollary}}

\newtheorem{assumption}{Assumption}

\newcommand{\DF}{\mathrm{D}\mathbf{F}}
\newcommand{\F}{\mathbf{F}}

\newcommand{\xx}{\mathbf{x}}

\newcommand{\zz}{\mathbf{z}}
\renewcommand{\AA}{\mathbf{A}}
\newcommand{\bb}{\mathbf{b}}

\newcommand{\FF}{\mathbf{F}}

\newcommand{\MM}{\mathbf{M}}
\newcommand{\yy}{\mathbf{y}}

\newcommand{\HH}{\mathbf{H}}
\renewcommand{\S}{\mathbf{S}}

\newcommand{\R}{\mathrm{R}}

\DeclareMathOperator*{\argmin}{arg\,min}
\DeclareMathOperator{\proj}{proj}
\DeclareMathOperator{\card}{card}

\begin{document}

\begin{frontmatter}

\title{Online Sketched Newton-Raphson} 

\author[jll]{Jean-Luc~Lupien}\ead{jllupien@berkeley.edu},\author[ymp]{Yuen-Man~Pun}\ead{yuenman.pun@anu.edu.au}, \author[yd]{Youssef~Diouane}\ead{youssef.diouane@polymtl.ca}, \author[is]{Iman~Shames}\ead{iman.shames@unimelb.edu.au}, and \author[all]{Antoine~Lesage-Landry}\ead{antoine.lesage-landry@polymtl.ca}

\address[jll]{Department of Civil and Environmental Engineering, University of California, Berkeley, CA, United States of America}  
\address[ymp]{School of Engineering, Australian National University, ACT, Australia}%
\address[yd]{Department of Mathematics and Industrial Engineering \& GERAD, Polytechnique Montr\'eal, QC, Canada}%
\address[is]{Department of Electrical and Electronic Engineering, University of Melbourne,  VIC, Australia}%
\address[all]{Department of Electrical Engineering, GERAD \& Mila, Polytechnique Montr\'eal, QC, Canada}%


\begin{abstract}                          

In online convex optimization (OCO), a decision-maker is confronted with an unknown environment and seeks to play an optimal sequence of decisions on a short time-scale using only past information. Recent advances in second-order OCO methods have demonstrated tighter regret bounds and improved empirical performance over traditional first-order methods. However, this performance comes at a cost: a matrix inversion is now required, which scales with the cube of the size of the problem. In this work, we propose sketching to mitigate this limitation. Specifically, we present the online sketched Newton-Raphson method (\texttt{OSNR}) which preserves the tight regret bounds obtained with second-order methods while presenting a strict computational improvement in terms of complexity.
 We discuss three application scenarios of \texttt{OSNR}: online root finding, unconstrained OCO, and time-varying equality-constrained OCO, and present their respective regret and a constraint violation bound for the latter. In all three applications, \texttt{OSNR} achieves sublinear dynamic regret bounds. For the equality-constrained case, the extension \texttt{OSNR} with equality constraints (\texttt{OSNR-EC}) is shown to yield sublinear cumulative constraint violation. Finally, we illustrate the performance of \texttt{OSNR} and \texttt{OSNR-EC} on two numerical examples, viz., online position tracking and optimal power flow, and observe that \texttt{OSNR} and \texttt{OSNR-EC} exhibit high performance even at low sampling rates.

\end{abstract}

\end{frontmatter}

\section{Introduction}

Online convex optimization (OCO) considers optimization as a process~\cite{hazanOCO}. In OCO, a decision-maker must act in an uncertain or potentially adversarial environment, using only prior information, to provide an optimal solution to the optimization problem at hand~\cite{zinkevich2003online}. Specific to OCO, the decision-maker is assumed to have to commit to a decision before the environmental loss is revealed. Additionally, it is assumed that limited computational resources and time are available to the decision-maker when facing new rounds, hence motivating the design of streamlined algorithms. The objective for the decision-maker is to minimize their regret, i.e., the cumulative distance from their incurred loss to its minimum at each time step. Specifically, when designing an OCO algorithm, one seeks to obtain a provable upper bound on its regret. The aim is to obtain a regret bound that is sublinear in time, as this implies optimal decisions are dispatched, on average, over a long time horizon. OCO has many applications, including in real-time operations of power grids, online optimal portfolio selection, and online recommendation systems~\cite{OIPM, hazanOCO, Shalev}. 

Sketched Newton-Raphson methods have improved the convergence rate and increased the domain of this family of root-finding algorithms~\cite{sketchedNR, NewtonSketch}. Sketched Newton-Raphson methods are distinctly amenable to online convex optimization because they address the fundamental limitation of prior second-order online algorithms: the time complexity of matrix inversion~\cite{SubsampledNewtonMethods, SketchandProject}. In the context of OCO, the computational constraint makes matrix inversion unrealistic for large-scale problems. However, the sketching process greatly reduces the computational resources required for inversion~\cite{SketchedSum, SubsampledNewtonMethods}. For this reason, this work proposes the online sketched Newton-Raphson (\texttt{OSNR}) for online root finding and second-order online optimization. We finally extend our approach to a time-varying linear equality-constrained setting.

Let $\xx_t\in\mathbb{R}^n$, $n\in\mathbb{N}$, be the decision vector at timestep $t=1,2,3,\dots,T$. Let $g_t:\mathbb{R}^n\mapsto \mathbb{R}$ be the convex loss function at each timestep, and $\mathcal{X}\subseteq \mathbb{R}^n$ be a compact, convex set. The standard OCO problem can then be expressed as:
\begin{equation}
    \min_{\xx_t\in\mathcal{X}}g_t(\xx_t).
\end{equation}
OCO algorithm's performance is characterized in terms of the regret which can be either static or dynamic. 
Static regret compares the loss suffered from the sequence of decisions to the best fixed decision in hindsight. 
In this work, we focus on dynamic regret as it is a stricter performance indicator and better suited for many engineering applications, e.g., tracking or cost minimization. Dynamic regret over a time horizon $T$, $\R(T)$, benchmarks the sequence of decisions against the round-optimal decision at each timestep. This is expressed as:
\begin{equation}
    \R(T) = \sum_{t=1}^T g_t(\xx_t)-g_t(\xx^*_t),
\end{equation}
where $\xx_t^*$ denotes the round-optimal decision at each timestep $t$: $\xx_t^* \in \argmin_{\xx\in \mathcal{X}} g_t(\xx)$. A bound on the dynamic regret, henceforth only referred as regret, implies a bound on the static regret, as the round-optimal decisions will always yield lower or equal loss than the best fixed decision in hindsight.

In this work, we first consider a slightly different context: online root-finding. By using the OCO framework, we establish provable performance guaranties adapted to the online setting. While the decision-vector $\xx_t$ is the same as in standard OCO problems, we now consider differentiable functions of the following form: $\F_t:\mathbb{R}^n\to \mathbb{R}^m$, $n,m\in\mathbb{N}$ as objectives. The problem to solve at each round then becomes:
\begin{equation}
    \min_{\xx_t \in \mathbb{R}^n} \norm{\F_t(\xx_t)},
\end{equation}
for $t=1,2,\dots,T$ and where $\norm{\cdot}$ denotes the Euclidean norm. Assuming that the function $\F_t$ has a non-empty set of zeros, the minimum will therefore always be $0$. The regret reduces simply to:
\begin{equation}
    \R_{\text{0}}(T) = \sum_{t=1}^T \norm{\F_t(\xx_t)}.
\end{equation}

\emph{Related work.}
The literature stream closest to \texttt{OSNR} is that of the second-order OCO methods. These algorithms tend to outperform their first-order counterparts, such as gradient descent~\cite{hazanOCO, OnlineNewton}, while also avoiding the potentially cumbersome projection step. For example, in~\cite{HazanInterior}, an interior-point method is presented, and a tight static regret bound is obtained. More recently,~\cite{OnlineNewton} has proposed an online Newton's method and has achieved a sublinear dynamic regret. This work is extended in~\cite{OPENM} and~\cite{OIPM} by incorporating time-varying linear and convex conic constraints, respectively, while conserving the sublinear dynamic regret bounds. OCO aims to provide streamlined decision rules to promote a high speed of computation given limited resources. However, the aforementioned second-order methods do require a costly matrix inversion step, which replaces the projection step as the update's computational bottleneck. 
The inversion has a complexity of $\mathcal{O}(n^3)$ where $n$ is the size of the matrix. This hints at a potential gain if the dimension of the matrix-to-be-inverted is effectively decreased.

Sketching presents a solution to this problem by sub-sampling the matrix before inverting it~\cite{RandomizedSubNewton}. This process can reduce the complexity of inversion to $\mathcal{O}(m\tau^2+\tau^3)$ where $\tau$ is the  sketch size to be defined in the next section~\cite{NewtonSketch}. In~\cite{luo2017efficientsecondorderonline}, the authors use a sketched version of Newton's method in an online optimization setting. The method displays good empirical results; however, no formal analysis is provided. As such, there is no theoretical bound on the regret for this approach.

Next, we provide background definitions and assumptions in Section~\ref{sec:background}. The online sketched Newton-Raphson step is characterized in Section~\ref{sec:step}. Our method for online root-tracking, OCO, and OCO subject to time-varying equality constraints are presented in~Section~\ref{sec:algo} together with their respective regret analysis. Numerical examples are discussed in Section~\ref{sec:eg} and concluding remarks are provided in~Section~\ref{sec:conc}.


\section{Background}\label{sec:background}

We start by providing the main definitions and assumptions of this work.

\subsection{Definitions}

We seek to find the zeros of a differential function $\F_t: \mathbb{R}^n\to \mathbb{R}^m$ with $n,m\in\mathbb{N}$. Let $\mathcal{X}^0_t = \left\{\xx\in \mathbb{R}^n \, | \, \F_t(\xx^*)=0\right\}$ be the set of zeros of $\F_t$ We suppose the set $\mathcal{X}^0_t$ to be non-empty for all $t$. We define the transpose of the Jacobian matrix of $\F_t$, $\DF_t : \mathbb{R}^n\to \mathbb{R}^{n\times m}$, as \[\DF_t(\xx) = \begin{bmatrix}
    \frac{\partial F_t(1)}{\partial x(1)} (\xx)&\frac{\partial F_t(2)}{\partial x(1)}(\xx)&\cdots&\frac{\partial F_t(m)}{\partial x(1)}(\xx)\\
    \frac{\partial F_t(2)}{\partial x(1)}(\xx)&\frac{\partial F_t(1)}{\partial x(2)}(\xx)&\cdots&\frac{\partial F_t(m)}{\partial x(n)}(\xx)\\
    \vdots&\vdots&\ddots&\vdots\\\frac{\partial F_t(1)}{\partial x(1)}(\xx)&\frac{\partial F_t(2)}{\partial x(2)}(\xx)&\cdots&\frac{\partial F_t(m)}{\partial x(n)}(\xx)
\end{bmatrix}.\]

For simplicity and numerical performance, we employ the uniform sub-sampling sketching matrices in this work~\cite{SketchedBasics}. These matrices are sampled from a distribution $\mathbb{D}_\tau$ with $\tau \in \{1,2,\ldots, n\} \equiv\llbracket 1, n  \rrbracket$ being the sketch size. Let $\S\in\mathbb{R}^{m\times \tau}$ denote the sketching matrix sampled from $\mathbb{D}_\tau$.
First, define the sketched Hessian function under sketching $\S$ at $t$ and $\xx$ as:
\[\HH^{\S}_t(\xx) = \S(\S^\top \DF_t(\xx)^\top \DF_t(\xx) \S)^\dagger \S^\top,\]
where $^\dagger$ represents the Moore-Penrose pseudo-inverse~\cite{MoorePenroseBasics}. Then, let $f^{\S,\yy}_t(\xx)$ denote the square of the function $\F_t$ in the sketched Hessian norm. This can be expressed as:
\[f^{\S,\yy}_t(\xx) = \frac{1}{2}\mathbf{F}_t(\xx)^\top \HH^{\S}_t(\yy)\mathbf{F}_t(\xx) = \frac{1}{2}\norm{\mathbf{F}_t(\xx)}_{\HH^{\S}_t(\yy)}^2.\]
Taking the expected value over the possible sketching matrices, we define $f^\yy_t (\xx)$ as:
\[f^{\yy}_t(\xx) = \mathbb{E}_\S\left[f^{\S,\yy}_t(\xx)\right].\]
Finally, we note the identity from \cite[Lemma 4.1]{sketchedNR} linking $f^{\S,\yy}(\xx)_t$ to its gradient:
\begin{equation}\label{gradIdent}\frac{1}{2}\norm{\nabla f^{\S,\yy}_t(\xx)}^2 = f^{\S,\yy}_t(\xx).\end{equation}
Hereinafter, the norm denotes the Euclidean norm. 

When tracking online solutions over time, we use the 
cumulative variation between zeros, $V_T$,  as:
\[V_T = \sum_{t=1}^{T} \textrm{H}_{\text{d}}(\mathcal{X}^0_t, \mathcal{X}^0_{t-1}),\]
where $\textrm{H}_{\text{d}}$ represents the Hausdorff distance between sets and $\mathcal{X}^0_t$ represents the set of zeros of the function $\F_t$ at consecutive timesteps. When the set of zeros is a singleton, i.e., $\mathcal{X}^0_t=\{\xx_t^*\}$, we retrieve the common definition of cumulative variation given by:
\[V_T = \sum_{t=1}^T \norm{\xx_t^*-\xx_{t+1}^*}.\]

\subsection{Assumptions}
We now state our two main assumptions.

\begin{assumption} \label{ass:lip} We consider that $\F_t$ is Lipschitz continuous for all $t$. Specifically, we suppose there exists $0<L_{\F_t} <+\infty$ such that:
\[\norm{\F_t(\xx)-\F_t(\yy)}\le L_{\F_t} \norm{\xx-\yy}\quad \forall \xx,\yy \in \mathbb{R}^n,  \forall t.\]\end{assumption}
This assumption is mild and common in the OCO literature stream~\cite{hazanOCO, Shalev}. We further define $\displaystyle L_\F = \max_{t=1,2,\ldots,T} L_{\F_t}$.
\begin{assumption} \label{ass:quasi} We suppose that $f^\xx_t(\xx)$ is a special case of (1,$\mu$)-strongly quasar-convex for all time $t$. Specifically, we assume that there exists $0<\mu_t<1$ such that~\cite{hermant2024study}:
\begin{equation}
\begin{aligned}
f^{\xx}_t (\xx^*) \ge& \; f^{\xx}_t(\xx) + \nabla f^\xx_t(\xx)^\top(\xx^*-\xx)\\
&\qquad \quad + \frac{\mu_t}{2}\norm{\xx^*-\xx}^2 \quad \forall \xx \in \mathbb{R}^n,  \forall t.
\end{aligned}\label{strconvex}
\end{equation}
\end{assumption}
Finally, we also define $\displaystyle\mu = \min_{t=1,2,\ldots,T} \mu_t$.

\section{Online Sketched Newton-Raphson}
\label{sec:step}
In this section, we introduce the sketched Newton-Raphson step and derive an online optimization algorithm leveraging it. We will also provide regret bounds for our approach when applied to different scenarios of online problems.


Given a sketch $\S_t \sim \mathbb{D}_\tau$ and an initial iterate $\xx_t$ at time~$t$, the sketched Newton-Raphson step is defined as:
\begin{equation*}
\begin{aligned}
    \label{NRStep}
    \xx_{t+1} &= \xx_t\\
    & - \DF_t(\xx_t)\S_t(\S_t^\top \DF_t(\xx_t)^\top \DF_t(\xx_t)\S_t)^\dagger \S_t^\top \F(\xx_t).
\end{aligned}    
\end{equation*}
The sketched Newton-Raphson step can be re-expressed as:
\begin{equation}
\xx_{t+1} = \xx_t-\nabla f^{\S_t,\xx_t}_t(\xx_t). \label{eq:skns}
\end{equation}

The computational complexity of this update is given by: $\mathcal{O}(m\tau+n\tau^2+\tau^3)$ \cite{sketchedNR}. This complexity is upper bounded by that of the standard Newton-Raphson update. 

We next establish the strict improvement, in the expected sense, of the sketched Newton-Raphson step~\eqref{eq:skns} under Assumption~\ref{ass:quasi}.
\begin{lem}
\label{lem:linred}

    Under Assumption 2, the sketched Newton-Raphson step~\eqref{eq:skns} exhibits a strict improvement with respect to a zero of the function $\F_t$. Specifically, the following holds:
\begin{equation*}
    \mathbb{E}\left[\norm{\xx_{t+1}-\xx_t^*}\Big]\le\sqrt{1-\mu} \;\mathbb{E}\Big[\norm{\xx_t-\xx_t^*}\right],
\end{equation*}
where the expectation is taken with respect to all sketch matrices $\{\mathbf{S}_t\}_{t=1}^T$.
\end{lem}

\begin{pf} Let $\bm{\delta}_t = \xx_t-\xx_t^*$. From the definition of the sketched Newton-Raphson step, we have:
\[
\begin{array}{l}
  \mathbb{E}\left[\norm{\xx_{t+1}-\xx_t^*}^2\right] = \mathbb{E}\Big[\norm{\xx_t-\nabla f^{\S_t,\xx_t}_{t}(\xx_t)-\xx_t^*}^2\Big]\\
    \quad = \norm{\bm{\delta}_t}^2 + \mathbb{E}\Big[-2\bm{\delta}_t^\top\nabla f^{\S_t,\xx_t}_t(\xx_t) +\norm{\nabla f^{\S_t,\xx_t}_t(\xx_t)}^2\Big] \\
     \quad = \norm{\bm{\delta}_t}^2 - 2\bm{\delta}_t^\top\nabla f^{\xx_t}_t(\xx_t)+\mathbb{E}\Big[\norm{\nabla f^{\S_t,\xx_t}_t(\xx_t)}^2\Big].
\end{array}
\]

Rearranging \eqref{strconvex}, we get:
\[\nabla f^{\xx_t}_t(\xx_t)^\top(\xx_t-\xx_t^*) \ge f^{\xx_t}_t(\xx_t)-f^{\xx_t}_t(\xx_t^*)+\frac{\mu_t}{2}\norm{\xx_t^*-\xx_t}^2.\]
Yet, we have that $\bm{\delta}_t^\top\nabla f^{\xx_t}_t(\xx_t) =  \nabla f^{\xx_t}_t(\xx_t)^\top({\xx_t}-\xx_t^*)$ yielding:
\[
\begin{array}{l}
 \mathbb{E}\left[\norm{\xx_{t+1}-\xx_t^*}^2\right] \\
 \quad \le (1-\mu_t)\norm{\bm{\delta}_t}^2 +\mathbb{E}\Big[\norm{\nabla f^{\S_t,\xx_t}_t(\xx_t)}^2] \\
 \qquad \qquad \qquad \qquad \qquad - 2\left(f^{\xx_t}_t(\xx_t)-f^{\xx_t}_t(\xx^*_t)\right)\\
 \quad =(1-\mu_t) \norm{\bm{\delta}_t}^2+2f^{\xx_t}_t(\xx_t) -2\left(f^{\xx_t}_t(\xx_t)-f^{\xx_t}_t(\xx^*_t)\right)\\
 \quad = (1-\mu_t) \norm{\bm{\delta}_t}^2+2f_{\xx_t}(\xx_t^*),
\end{array}
\]

where the second to last equation follows from~\eqref{gradIdent}. Finally, we observe that $f_{\xx_t}(\xx_t^*)=0$ by assumption, given that $\xx_t^*$ is such that $\F_t(\xx_t^*)=\mathbf{0}$ by definition. Therefore, we have:
\begin{align*}
    \mathbb{E}\Big[\norm{\xx_{t+1}-\xx_t^*}^2\Big] &\le (1-\mu_t) \norm{\bm{\delta}_t}^2\\
    \iff \mathbb{E}\Big[\norm{\xx_{t+1}-\xx_t^*}\Big]&\le\sqrt{1-\mu_t}\norm{\bm{\delta_t}}\\
    &\le \sqrt{1-\mu_t}  \mathbb{E}\Big[\norm{\xx_t-\xx_t^*}\Big],
\end{align*}
where we use the concavity of the square root to establish that $\mathbb{E}\Big[\norm{\xx_{t+1}-\xx_t^*}\Big] \le \sqrt{\mathbb{E}\Big[\norm{\xx_{t+1}-\xx_t^*}^2\Big]}$. 
Lower bounding $\mu_t$ by $\mu$ completes the proof. \hfill $\square$
\end{pf}
Although the initial assumption of strong quasi-convexity is quite strong, this result is equally impactful in scope as we obtain an expected linear convergence for the Newton-Raphson step. More importantly, this convergence is realized over the entire domain of the function $\F_t$. Given a tolerance and an initial point, we can calculate the exact number of steps before we expect to reach a solution under this tolerance.

\section{Algorithms}\label{sec:algo}

We now present the main algorithm of this work: the Online Sketched Newton-Raphson (\texttt{OSNR}) method. The idea behind the algorithm is to repeatedly use the sketched Newton-Raphson step in an online fashion. The \texttt{OSNR} is presented in Algorithm~\ref{alg:ONR}. 

\begin{algorithm}[h!]
\caption{The Online Sketched Newton-Raphson (\texttt{OSNR}) method}\label{alg:ONR}
\begin{algorithmic}[1]
\State\textbf{Initialization}: Receive $\xx_0 \in\mathbb{R}^n$ and $\tau \in \llbracket 1, n  \rrbracket$  the sketching size.
\FOR{$t=1,2,3...T$}
\State Implement the decision $\xx_t$.
\State Observe the outcome $\F_t(\xx_t)$. 
\State Sample a new sketching matrix $\S_t \sim \mathbb{D}_\tau$.
\State  Update decision: $\xx_{t+1} = \xx_t - \nabla f_{\S_t,\xx_t}(\xx_t)$.
\ENDFOR
\end{algorithmic}
\end{algorithm}

We remark that an immediate improvement to this method could be to perform multiple sketched Newton-Raphson steps at each iteration, given its fast convergence properties~\cite{sketchedNR, NewtonSketch}. This could improve regret when applied in practice. Given the expected speedup in calculation time when compared to a deterministic Newton-like step, these iterations could be made within the computational and temporal constraints. To remain closely aligned with OCO, we restrict \texttt{OSNR} to one step per time $t$. A multiple sketching step approach is a topic for future work. 

The following sections present different application scenarios for \texttt{OSNR}. 

\subsection{Online Root-finding}\label{ssec:online_zeroes}

We now apply the \texttt{OSNR} method to the online tracking of zeros problem and bound its regret.
Recall that the problem takes the form:
\begin{equation}
    \min_{\xx_t} \; \norm{\F_t(\xx_t)}, \label{eq:zeros}
\end{equation}
at each time step $t=1,2,\dots,T$. Then, using \texttt{OSNR} yields the following.

\begin{theorem}
    Under Assumptions \ref{ass:lip} and \ref{ass:quasi}, the expected regret suffered by Algorithm~\ref{alg:ONR} when solving~\eqref{eq:zeros} is bounded above by:
    \begin{equation*}
        \mathbb{E}\left[\R_\textup{0} (T)\right] \le  \frac{L_{\F} }{1-\sqrt{1-\mu}}\left(V_T + C\right) = \mathcal{O}(V_T+1),
    \end{equation*}
    where $C=\sqrt{1-\mu}\norm{\xx_0-\xx_0^*}$.
\end{theorem}

\begin{pf}
The Lipschitz continuity from Assumption 1 leads to
\begin{align}
    \R_{\text{0}}(T) &= \sum_{t=1}^T\norm{\F_t(\xx_t)-\F_t(\xx_t^*)}\nonumber\\
    &\le \sum_{t=1}^T L_{\F_t} \norm{\xx_t-\xx_t^*} \label{eq:substi}\\
    &\le L_{\F_t} \sum_{t=1}^T\norm{\xx_t-\xx^*_{t-1}}+L_{\F_t} \sum_{t=1}^T \norm{\xx_{t-1}^*-\xx_t^*}. \nonumber
\end{align}
Taking the expectation on both sides with respect to the sketching matrices we get:
\[
\begin{array}{l}
    \mathbb{E}\left[\R_{\text{0}}(T)\right] \le L_{\F_t} \mathbb{E}\left[\displaystyle\sum_{t=1}^T\norm{\xx_t-\xx^*_{t-1}}\right] + L_{\F_t} V_T\\
    \qquad \le L_{\F_t} \displaystyle\sum_{t=0}^{T-1} \sqrt{1-\mu}\mathbb{E}\left[\norm{\xx_t-\xx_t^*}\right] + L_{\F_t} V_T\\
    \qquad = L_{\F_t} \mathbb{E}\left[\displaystyle\sum_{t=1}^{T} \sqrt{1-\mu}\norm{\xx_t-\xx_t^*}\right]  +L_{\F_t} V_T\\
    \qquad  \quad +L_{\F_t} \sqrt{1-\mu}( \norm{\xx_0-\xx_0^*}-\mathbb{E}\left[L_{\F_t} \norm{\xx_T-\xx_T^*}\right]),
\end{array}
\]
where we used Lemma~\ref{lem:linred} to obtain the second inequality.
Noting that $( \sqrt{1-\mu}\norm{\xx_0-\xx_0^*}-\mathbb{E}\left[\norm{\xx_T-\xx_T^*}\right]) \leq \sqrt{1-\mu}\norm{\xx_0-\xx_0^*}= C$, we have that:
\begin{align*}
     \mathbb{E}\left[L_{\F_t} \sum_{t=1}^T \norm{\xx_t-\xx_t^*}\right]&\le L_{\F_t} \mathbb{E}\left[\sum_{t=1}^{T} \sqrt{1-\mu}\norm{\xx_t-\xx_t^*}\right]\\
     &\qquad\qquad+L_{\F_t} V_T+L_{\F_t}  C,
\end{align*}
and thus,
    \begin{align*}
    (1-\sqrt{1-\mu})\mathbb{E}\left[L_{\F_t} \sum_{t=1}^T \norm{\xx_t-\xx_t^*}\right]&\le L_{\F_t} \left( V_T+ C \right).
\end{align*}
Substituting this in~\eqref{eq:substi} yields
\begin{align*}    
    \mathbb{E}\left[\R_{\text{0}}(T)\right]\le \frac{L_{\F_t} }{1-\sqrt{1-\mu}}\left(V_T + C\right),
\end{align*}

which results in the following regret bound:
\[ \mathbb{E}\left[\R_{\text{0}}(T)\right] \le  \frac{L_{\F} }{1-\sqrt{1-\mu}}\left(V_T + C\right) = \mathcal{O}(V_T+1),\]
where $L_\F$ is used to upper bound $L_{\F_t}$ for all $t$. \hfill $\square$

\end{pf}
The online root-finding algorithm extends the framework of OCO to a slightly different application. The regret bound implies that the cumulative error incurred is proportional to the temporal variation in roots.

\subsection{Online Convex Optimization with \texttt{OSNR}}

Now consider the unconstrained OCO problem described at time $t$ by:
\begin{align}
\label{eq:OCO}
    \min_{\xx_t}\; g_t(\xx_t),
\end{align}
where $g_t:\mathbb{R}^m\mapsto\mathbb{R}$ is a convex, twice-differentiable, and time-varying function. Letting $\mathbf{F}_t(\xx) = \nabla g_t(\xx)$, then tracking zeros of $\mathbf{F}_t(\xx)$ similarly to Section~\ref{ssec:online_zeroes} corresponds to computing first-order extrema of the function~$g_t$. For convex functions, this is equivalent to tracking global optima. We can hence apply Algorithm~\ref{alg:ONR} to OCO problems akin to~\eqref{eq:OCO} and obtained a bounded the regret. This is shown next.

\begin{corollary}
\label{cor:OCO}
Consider~\eqref{eq:OCO} and implement Algorithm~\ref{alg:ONR} with $\F_t(\xx) = \nabla g_t(\xx)$ for all time $t$. If $g_t$ is Lipchitz continuous with modulus $0< L_{g_t} < +\infty$ for all $t$ and is such that Assumption~\ref{ass:quasi} holds, then the expected regret incurred when solving \eqref{eq:OCO} is bounded by:
    \[
    \mathbb{E}\left[\R(T)\right] \le L_{g} \frac{V_T + \widehat{C}}{1-\sqrt{1-\mu}},
    \]
    where $\displaystyle L_{g} = \max_{t=1,2,\ldots,T} L_{g_t}$ and $\widehat{C}=\norm{\xx_1-\xx_1^*}$.
\end{corollary}

\begin{pf}
    Applying the dynamic regret definition to Problem \eqref{eq:OCO}, we get:
\[\R(T) = \sum_{t=1}^T g_t(\xx_t)-g_t(\xx_t^*).\]

The Lipschitz continuity of $g_t$ leads to:
\begin{equation*}
\begin{array}{l}
   \mathrm{R}(T) \le \displaystyle \sum_{t=1}^T L_{g_t}\norm{\xx_t-\xx_t^*} \\
   \qquad \le L_{g}\left[\displaystyle\sum_{t=2}^T \norm{\xx_t-\xx_{t-1}^*+\xx_{t-1}^*-\xx_t^*} + \norm{\xx_1-\xx_1^*}\right],
\end{array}
\end{equation*}

where we bound $L_{g_t}$ with $L_{g}$ to factor it from the sum. Thus,
\begin{equation}
\label{eq:regret_and_xtstar:1}
  \text{R}(T)\le L_{g}\displaystyle \sum_{t=2}^T\norm{\xx_t-\xx_{t-1}^*}+L_{g}V_T +L_{g}\norm{\xx_1-\xx_1^*}.
\end{equation}
Taking the expectation on both sides and using Lemma~\ref{lem:linred}, we obtain:
\begin{equation*}
\begin{array}{l}
\mathbb{E}\left[\displaystyle\sum_{t=1}^T \norm{\xx_t-\xx_t^*}\right] \\
    \qquad\le \mathbb{E}\left[ \displaystyle\sum_{t=2}^T \norm{\xx_t-\xx_{t-1}^*}\right] + V_T + \norm{\xx_1-\xx_1^*}\\
       \qquad =  \displaystyle\sum_{t=2}^T \mathbb{E}\left[\norm{\xx_t-\xx_{t-1}^*}\right]  + V_T + \norm{\xx_1-\xx_1^*} \\
       \qquad\le\displaystyle\sum_{t=2}^T \sqrt{1-\mu}\mathbb{E}\left[\norm{\xx_{t-1}-\xx_{t-1}^*}\right] + V_T + \norm{\xx_1-\xx_1^*}
       \end{array}
\end{equation*}
\begin{equation*}
\begin{array}{l}    
  \qquad= \sqrt{1-\mu}\mathbb{E}\left[ \displaystyle \sum_{t=1}^{T-1} \norm{\xx_{t}-\xx_{t}^*}\right]+ V_T +\norm{\xx_1-\xx_1^*}.
\end{array}
\end{equation*}
 Hence, we have
\begin{equation}\label{eq:inter}
\begin{aligned}
    \left(1-\sqrt{1-\mu}\right) \mathbb{E}\left[\sum_{t=1}^T \norm{\xx_t-\xx_t^*}\right] &\leq V_T \\
    &\hspace{-3.5cm}+ \norm{\xx_1-\xx_1^*} -\sqrt{1-\mu}\mathbb{E}\left[\norm{\xx_T-\xx^*_T}\right]. 
\end{aligned}
\end{equation}
Substituting~\eqref{eq:inter} in the right-hand side of~\eqref{eq:regret_and_xtstar:1} taken in the expected sense, we finally bound $\mathbb{E}\left[\text{R}_d(T)\right]$. Defining $\widehat{C}=\norm{\xx_1-\xx_1^*} \geq \norm{\xx_1-\xx_1^*} -\sqrt{1-\mu}\mathbb{E}\left[\norm{\xx_T-\xx^*_T}\right]$, we obtain:
\begin{align}
    \mathbb{E}\left[\text{R}(T)\right] &\le L_{g} \frac{V_T + \widehat{C}}{1-\sqrt{1-\mu}},
\end{align}
which concludes the proof. \hfill $\square$
\end{pf}

We further extend this result to the equality-constrained case in the next section before discussing both Corollary~\ref{cor:OCO} and the next section's Corollary~\ref{cor:OCO=}.

\subsection{Time-varying Linear Equality-Constrained OCO}

Next, we add time-varying linear equality constraints to the Problem~\eqref{eq:OCO} and show that \texttt{OSNR} can be employed to solve the problem. Consider the affine equality-constrained online optimization problem:
\begin{subequations}\label{eq:consOCO}
\begin{align}
\min_{\xx_t}\; &g_t(\xx_t) \label{eq:consOCO_ojb}\\
    \text{s.t.}\;&\AA \xx_t = \bb_t, \label{eq:consOCO_const}
\end{align}
\end{subequations}
where $\AA \in \mathbb{R}^{k\times n}$ and $k\in \mathbb{N}$ are constant parameters, $\bb_t\in \mathbb{R}^{k}$ is an online parameter observed similarly to~$g_t$. With the inclusion of a time-varying equality constraints, we define the cumulative constraint violations:
\begin{equation}\label{eq:vio}
\text{Vio}(T) = \sum_{t=1}^T \norm{\AA \xx_t - \bb_t}.    
\end{equation}
This definition is akin to the regret but for online equality constraints satisfaction. It can be further extended to account for online inequality constraints when needed~\cite{yi2022regret}.

Problem~\eqref{eq:consOCO} can be solved as an unconstrained optimization problem akin to~\eqref{eq:OCO} by always starting an update at a feasible point and then limiting the search directions to those that are in the nullspace of the constraint matrix $\AA$. 
Specifically, consider (i) a matrix with orthonormal columns $\MM\in\mathbb{R}^{n \times (n-k)}$ such that $\AA\MM\zz =\mathbf{0}$ for all $\zz \in \mathbb{R}^{n-k}$ and (ii) a point $\tilde{\xx}_t$ such that $\AA\tilde{\xx}_t=\bb_t$. Then, we can equivalently solve Problem~\eqref{eq:consOCO} as:
\begin{equation}
\label{eq:reduced}
    \min_{\zz_t}\; g_t(\MM\zz_t+\tilde{\xx}_t).
\end{equation}
The optimal solution to the original problem is recovered from $\zz^*_t \in \argmin_{\zz_t} g_t(\MM\zz_t+\tilde{\xx}_t)$ via $\xx_t^* = \MM\zz^*_t+\tilde{\xx}_t$ \cite{boyd}. \texttt{OSNR} can be modified and applied to the reduced problem~\eqref{eq:reduced} using $\FF_t(\zz_t) =\MM^\top\nabla g_t(\MM \zz_t+\tilde{\xx}_t)$. Following a sketched Newton-Raphson step, the decision $\xx_{t+1}$ will not be feasible with respect to the newly observed constraint $\mathbf{b}_t$. 

A projection step onto the linear constraint hyperplane must first be completed, yielding the intermediary point $\tilde{\xx}_t$ as the update starts. The extended algorithm is presented in Algorithm~\ref{alg:ONRcons} as the Online sketched Newton-Raphson method with equality constraints (\texttt{OSNR-EC}). We note that we have purposely kept~$\mathbf{z}$ on Line 7 even if it is evaluated at~$\mathbf{0}$ to emphasize that the structure of the update.

\begin{algorithm}[tb]
\caption{Online sketched Newton-Raphson Method with Equality Constraints (\texttt{OSNR-EC})}\label{alg:ONRcons}
\begin{algorithmic}[1]
\State\textbf{Initialization}: Receive $\xx_0 \in\mathbb{R}^n, \AA\in\mathbb{R}^{k\times n}, \MM\in\mathbb{R}^{(n-k)\times n}$, and $\tau \in \llbracket 1, n  \rrbracket$  the sketching size.
\FOR{$t=1,2,3,...,T$}
\State Implement the decision $\xx_t$.
\State Observe the outcome $g_t(\xx_t)$ and parameter $\bb_t$. 
\State Sample a new sketching matrix $\S_t \sim \mathbb{D}_\tau$.
\State Project: $\tilde{\xx}_{t+1} = \xx_t + \AA^\top(\AA\AA^\top)^{-1}(\bb_{t}-\bb_{t-1})$.

\State  Update decision: 
\begin{align*}
\Delta \zz_t &= - \nabla f^{\S_t,\zz}_t(\MM\zz+\tilde{\xx}_{t+1})|_{\zz=\mathbf{0}},\\
\xx_{t+1} &= \tilde{\xx}_{t+1} + \MM \Delta\zz_{t}.
\end{align*}
\ENDFOR
\end{algorithmic}
\end{algorithm}

To establish a regret and constraint bound for Algorithm~\ref{alg:ONRcons}, we first provide two lemmas. First, we note that the projection step does not increase the distance from the optimal solution.
\begin{lem}\cite[ Theorem 3]{OPENM}
\label{lem:proj}
Let $\xx_t \in \mathbb{R}^n$ and consider $\tilde{\xx}_{t+1} = \proj_{\mathbf{A}\xx = \bb_t} \xx_t$. Then,
    \begin{align*}
        \norm{\tilde{\xx}_{t+1}-\xx_t^*} \le \norm{\xx_t-\xx_t^*}.
    \end{align*}
\end{lem}

Using this property, we can then establish the expected improvement of an \texttt{OSNR-EC} step.

\begin{lem}
\label{lem:eqred}
Consider Problem~\eqref{eq:consOCO} and implement an update of Algorithm~\ref{alg:ONRcons} with $\FF_t(\zz_t)=g_t(\MM\zz_t+\tilde{\xx}_t)$ for all~$t$ such that Assumption~\ref{ass:quasi} holds. Let there be a unitary matrix $\MM$ such that $\AA\MM\zz=0$, $\forall \zz\in \mathbb{R}^{n-k}$. Then the following holds:
    \begin{equation*}
\mathbb{E}\left[\norm{\xx_{t+1}-\xx_t^*}\right] \le \sqrt{1-\mu} \mathbb{E}\left[\norm{\xx_t-\xx_t^*}\right].
\end{equation*}
\end{lem}

\begin{pf} Let $\bm{\Delta}\mathbf{z}^*_t \in \mathbb{R}^{n-k}$ be such that $\mathbf{x}_t^* = \tilde{\xx}_t + \mathbf{M}\bm{\Delta}\mathbf{z}^*_t$, where $\mathbf{x}_t^*$ is an optimum of~\eqref{eq:consOCO}.
Re-expressing the difference between the computed point and the optimal solution we have:
\begin{align*}
    \norm{\xx_{t+1}-\xx_t^*} &= \norm{\MM\Delta\zz_t+\tilde{\xx}_{t+1}-\MM\Delta \zz_t^* -\tilde{\xx}_{t+1}}\\
    &= \norm{\MM(\Delta\zz_t-\Delta\zz_t^*)}\\
    &=\norm{\Delta\zz_t-\Delta\zz_t^*},
\end{align*}
where we obtained the last equality from the unitary matrix $\MM$ satisfying $\norm{\MM}=1$.
In particular, this is true in the expected sense meaning $\mathbb{E}\left[\norm{\xx_{t_1}-\xx_t^*}\right]=\mathbb{E}\left[\norm{\Delta \zz_t-\Delta\zz_t^*}\right]$. If the original function $\mathbf{F}_t(\xx)=\nabla g_t(\xx_t)$ satisfies Assumption \ref{ass:quasi}, then the reduced function $\mathbf{F}^{\text{red}}_t(\zz)=\MM^\top\nabla g_t(\MM\zz+\tilde{\xx}_t)$ also satisfies Assumption~\ref{ass:quasi} because the restriction of a quasi-convex function to a linear subspace conserves the quasi-convexity~\cite{renegar}. Consequently, Lemma~\ref{lem:linred} holds for $\mathbf{F}^{\text{red}}_t(\zz)$. Let $\zz^*_t$ be such that $\xx_t^* = \MM\zz_t^*$. Let $\zz^+$ denote the point resulting from a sketched Newton-Raphson step taken at $\zz$. Notice that under Algorithm~\ref{alg:ONRcons}, $\zz$ is always set to $\mathbf{0}$. Therefore, a sketched Newton step from this point respects:
\begin{align}
\label{eq:zRed}
    \mathbb{E}\left[\norm{\zz^+_t-\zz_t^*}\right]\le \sqrt{1-\mu}
    \mathbb{E}\left[\norm{\zz-\zz_t^*}\right]
\end{align}
Let $\Delta \zz_t=\zz^+_t-\zz$ and $\Delta \zz_t^*=\zz^*_t-\zz$. We notice that $\Delta \zz_t = \zz^+_t$ and $\Delta \zz_t^* = \zz_t^*$ because $\zz=0$ by definition. We can therefore re-express \eqref{eq:zRed} in the following way:

\begin{align}
\mathbb{E}\left[\norm{\Delta\zz_t-\Delta\zz_t^*}\right] &\le \sqrt{1-\mu}\mathbb{E}\left[\norm{\Delta\zz_t^*}\right]\nonumber\\
    &=\sqrt{1-\mu}\mathbb{E}\left[\norm{\MM\Delta\zz^*_t}\right]\nonumber\\
    &= \sqrt{1-\mu}\mathbb{E}\left[\norm{\MM\Delta\zz^*_t+\tilde{\xx}_{t+1}-\tilde{\xx}_{t+1}}\right]\nonumber\\
    &= \sqrt{1-\mu}\mathbb{E}\left[\norm{\tilde{\xx}_{t+1}-\xx_t^*}\right].\label{eq:projRed}
\end{align}
Using Lemma~\ref{lem:proj}, we can further bound \eqref{eq:projRed} by:
\begin{align*}
\mathbb{E}\left[\norm{\tilde{\xx}_{t+1}-\xx_t^*}\right]\le \mathbb{E}\left[\norm{\xx_t-\xx_t^*}\right].
\end{align*}
This concludes the proof.
\hfill $\square$
\end{pf}

Using Lemma \ref{lem:eqred}, we can now bound the expected regret and constraint violation of Algorithm \ref{alg:ONRcons}.

\begin{corollary}
\label{cor:OCO=}
    Consider Problem~\eqref{eq:consOCO} and implement Algorithm~\ref{alg:ONRcons} with $\FF_t(\zz_t)=g_t(\MM\zz_t+\tilde{\xx}_t)$ for all $t$, and a unitary matrix $\MM$ such that $\AA\MM\zz=0$, $\forall \zz\in \mathbb{R}^{n-k}$. If~$g_t$ is Lipschitz continuous with modulus $0< L_{g_t} < +\infty$, the expected regret suffered by Algorithm~\ref{alg:ONRcons} is bounded by:
\begin{align*}
    \mathbb{E}\left[\R(T)\right] &\le L_{g} \frac{V_T + \widehat{C}}{1-\sqrt{1-\mu}},
\end{align*}
where $\widehat{C} = \norm{\xx_1-\xx_1^*}$ and $\displaystyle L_{g} = \max_{t=1,2,\ldots,T} L_{g_t}$, and the constraint violation suffered is exactly described by:
 \begin{align*}
        \textup{Vio}(T) &=V_b(T),
    \end{align*}
    where $V_b(T) = \sum_{t=1}^T \norm{\mathbf{b}_t-\mathbf{b}_{t-1}}$ is the cumulative variation in online parameter $\mathbf{b}_t$.
\end{corollary}

\begin{pf}
From the definition of regret we get:
\[\R(T) = \sum_{t=1}^T g_t(\xx_t)-g_t(\xx_t^*).\]
The Lipschitz continuity of $g_t$ leads to:
\begin{align*}
    \R(T) &\le \sum_{t=1}^T L_{g_t}\norm{\xx_t-\xx_t^*} ,
\end{align*}
and equivalently,
\begin{align*}
    \mathbb{E}\left[\text{R}(T)\right]&= L_{g}\sum_{t=1}^T \mathbb{E}\left[\norm{\xx_t-\xx_t^*}\right].
\end{align*}
Using Corollary~\ref{cor:OCO}'s proof technique with Lemma~\ref{lem:eqred} instead of Lemma~\ref{lem:linred}, the regret is upper bounded by:
\begin{align}
    \mathbb{E}\left[\text{R}(T)\right] &\le L_{g} \frac{V_T + \widehat{C}}{1-\sqrt{1-\mu}},
\end{align}
with $\widehat{C}=\norm{\xx_1-\xx_1^*}$ and $\displaystyle L_{g} = \max_{t=1,2,\ldots,T} L_{g_t}$.

As for constraint violation, we recall the definition~\eqref{eq:vio}.

By construction, from the update's projection step (Line~6), we have that $\AA\xx_t = \bb_{t-1}$. Therefore, the round constraint violation is bounded by the change in online parameter $\bb_t$, leading to
    \begin{align*}
        \text{Vio}(T) &= \sum_{t=1}^T \norm{\bb_{t-1}-\bb_t} = V_b(T),
    \end{align*}
   which concludes the proof. \hfill $\square$
\end{pf}

\texttt{OSNR-EC} achieves $\mathcal{O}(V_T+1)$ bounds on both the regret similarly to \texttt{OSNR} and the constraint violation when applied to solving time-varying linear equality-constrained OCO problems. This indicates that under sublinear variation in optima, \texttt{OSNR} and \texttt{OSNR-EC} implement, in the expected sense, optimal, feasible decisions on average when the time horizon grows. The expected regret bound matches the tightest in the literature established in~\cite{OPENM, OIPM} for second-order methods
while permitting the use of a sub-sampled second-order matrix term. Additionally, \texttt{OSNR} has the same $\mathcal{O}(V_T+1)$ regret bound and computational complexity as a first-order method for strongly convex problem~\cite{SigmaOGD} due to the matrix multiplications. The latter significantly reduces the computational load when the matrix inversion/system of equations is done to perform the update.


\section{Numerical Examples}\label{sec:eg}
We next illustrate \texttt{OSNR} for 
OCO followed by its equality-constrained extension, \texttt{OSNR-EC}, in numerical examples.


\subsection{Target tracking}\label{ssec:target}
We apply \texttt{OSNR} to the
the localization of a moving target. We consider a target at a location $\yy_t\in \mathbb{R}^n$ at time $t$ and $m$ sensors. Each sensor $i=1,2,...,m$ is located at point $\mathbf{a}_i\in \mathbb{R}^n$. Suppose each sensor can measure the distance $d_i > 0$ between itself and the target $\yy_t$ as:
\begin{align*}
    d_i = \norm{\yy_t-\mathbf{a}_i},
\end{align*}
the online target tracking problem, at each time $t=1,2,...,T$, can be expressed as the following nonlinear, least-squares problem:
\begin{align}
\min_{\xx_t}\; \sum_{i=1}^m (\norm{\xx_t-\mathbf{a}_i}-d_i)^2.
\end{align}
In this example, we set $n=200$ and $m=180$. The locations of the sensors are fixed as $a_i = 20\mathcal{N}(\mathbf{0},\mathbf{I})$, where $\mathcal{N}(\mathbf{0},\mathbf{I})$ is the standard multivariate normal distribution. The time dynamics for the target are described by: $\yy_{t+1} = \yy_t+\mathbf{v}_t,$ where $\mathbf{v}_t$, the variation between rounds, is sampled as:
$\mathbf{v}_t = \frac{20\mathrm{N}(\mathbf{0},\mathbf{I})}{\sqrt{t}}$.

Multiple sketch sizes $\tau$ are utilized for the problem. The sketch sizes are re-expressed as the percentage of the total dimension (number of variables) of the problem, e.g., a sketch percentage of $\rho = 20\%$ means $\tau = \lfloor \rho n\rfloor = \lfloor 0.2 \cdot n\rfloor$. For each sketch percentage, $100$ executions of a time horizon of $T=1000$ steps are performed. The average regret for each is then calculated. As a comparison, the online gradient descent (\texttt{OGD}) algorithm is also implemented~\cite{zinkevich2003online} with stepsize set to $\eta=\frac{1}{15\sqrt{T}}$.

 \begin{figure}[htbp]
  \centering
  \includegraphics[width=0.9\textwidth]{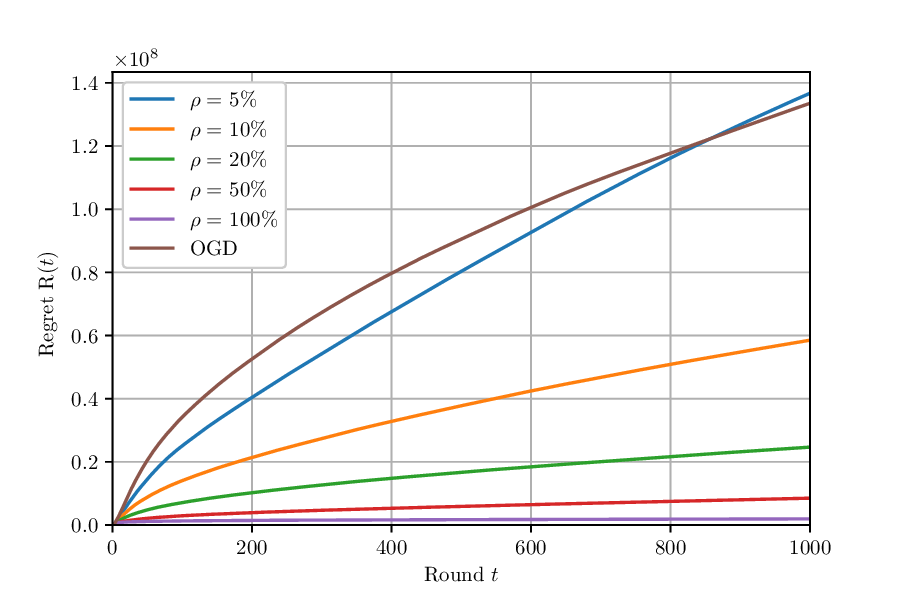} 
 \caption{Average regret for the online target tracking problem with \texttt{OSNR}.}
 \label{regPoints}
 \end{figure}

The experimental average regret for online tracking using~\texttt{OSNR} is presented in Figure~\ref{regPoints} for different sketch percentages. We observe that~\texttt{OSNR} outperforms \texttt{OGD} for all $\rho$. For a $\rho=5\%$-sketch percentage, the performance is similar, but ultimately better for~\texttt{OSNR}. A sketch percentage of 100\% results in using the entire Hessian matrix at every round. At this time, the update coincides with the online Newton method (\texttt{ONM}) from~\cite{OnlineNewton}. The average executing time of \texttt{OSNR}'s update is presented Figure~\ref{fig:timePoints} together with \texttt{OGD}'s. It illustrates that the computation time becomes similar to a first-order method at low sketch percentage while yielding improved regret. 


 \begin{figure}[htbp]
  \centering
  \includegraphics[width=0.85\textwidth]{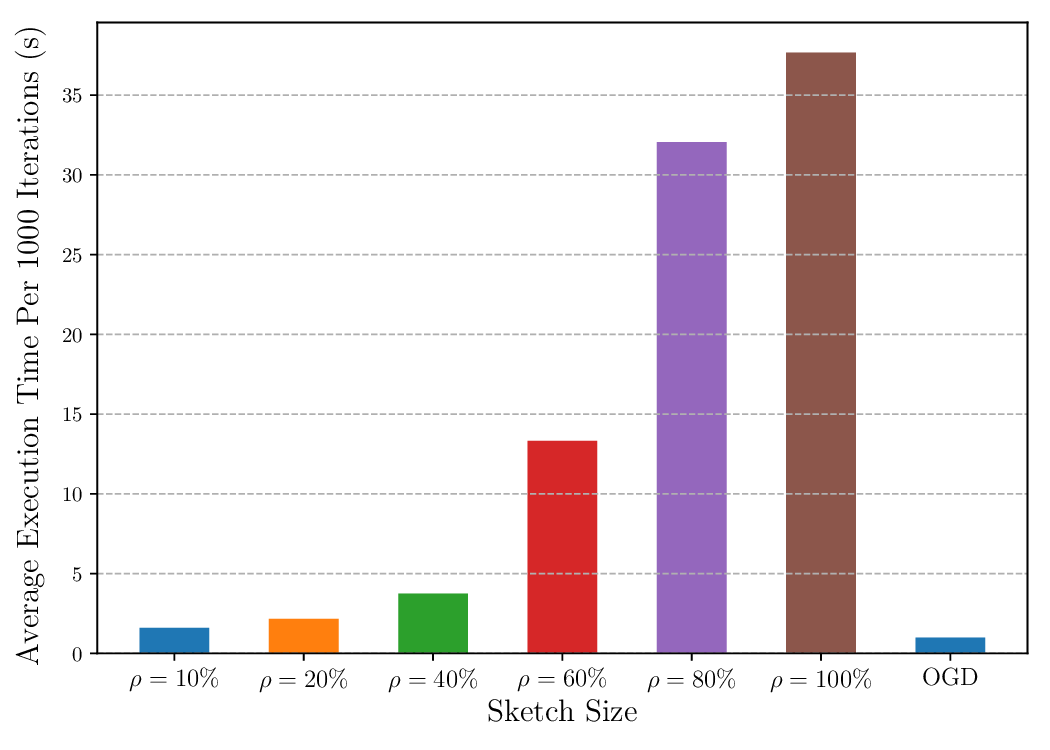} 
 \caption{Comparison of the \texttt{OSNR}'s average execution times.}
 \label{fig:timePoints}
 \end{figure}

\subsection{Online Linearized Optimal Power Flow}
We next consider the the linearized optimal power flow (DC-OPF)~\cite{taylor2015convex} and employ the equality-constrained \texttt{OSNR-EC} method to solve it in an online fashion. We suppose a network over the set of nodes $\mathcal{B} \subset \mathbb{N}$ and of lines $\mathcal{L} \subset \mathcal{B} \times \mathcal{B}$ with generators $\mathcal{G}\subseteq \mathcal{B}$. Let $\mathbf{p} \in \mathbb{R}^{\card \mathcal{B}}$, $\mathbf{P}\in \mathbb{R}^{\card \mathcal{L} \times \card \mathcal{L}}$, and $\bm{\theta}\in \mathbb{R}^{\card \mathcal{B}}$ be the nodal power injection/absorption, the line power flow, and the nodal phase angle. The problem at time $t$ takes the form:
\begin{subequations}
\begin{align}
    \min_{\mathbf{p}_t, \mathbf{P}_t, \bm{\theta}_t}\quad& \sum_{i\in \mathcal{G}} a_i p_{t,i}^2+b_i p_{t,i}\label{eq:dcopf_cost}\\
    \text{s.t.}\;\;\,\quad& \sum_{ij\in\mathcal{L}} p_{t,ij} = p_{t,i}, \quad i \in \mathcal{B} \label{eq:dcopf_nodal}\\
    & P_{t,ij} = B_{ij}(\theta_{t,i}-\theta_{t,j}), \quad ij \in \mathcal{L}\label{eq:dcopf_pf}\\
    & P_{t,ij} \le \overline{s}_{ij}, \quad ij \in \mathcal{L} \label{eq:dcopf_thermal}\\
    & \underline{p}_i \le p_{t,i} \le \overline{p}_i,\quad i \in \mathcal{B},\label{eq:dcopf_gen}
\end{align}
\end{subequations}
where~\eqref{eq:dcopf_cost} is the generation cost with coefficients $a_i>0$ and $b_i \in \mathbb{R}$ for all $i\in \mathcal{G}$,~\eqref{eq:dcopf_nodal} is the nodal power balance,~\eqref{eq:dcopf_pf} is the linearized power flow with $B_{ij} \in \mathbb{R}$ the line susceptance,~\eqref{eq:dcopf_thermal} enforces the thermal line rating $\overline{S}_ij$, and~\eqref{eq:dcopf_gen} imposes the generation limits $\underline{p}_i$ and $\overline{p}_i$ with $\underline{p}_i = \overline{p}_i = d_{t,i} < 0$ for a load $i \in \mathcal{B}\setminus\mathcal{G}$ with demand~$d_{t,i}$.
The inequalities cannot be solved using \texttt{OSNR-EC}. The following related problem is considered:
\begin{subequations}
\begin{align}
    \min_{\mathbf{p}_t, \mathbf{P}_t, \bm{\theta}_t}\quad& \sum_{i\in \mathcal{G}} a_i p_{t,i}^2+b_i p_{t,i} + \sum_{ij\in\mathcal{L}} \exp(\alpha_{ij} P_{t,ij}^2)\label{eq:dcopf_newcost} \\
    \text{s.t.}\;\;\,\quad& \sum_i P_{t,ij} = p_{t,i},\quad i\in \mathcal{G}\\
    &\sum_i P_{ij} = d_i,\quad  i \in \mathcal{B}\setminus\mathcal{G}\label{eq:dcopf_demand}\\
    & P_{ij} = B_{ij}(\theta_i-\theta_j), \quad i \in \mathcal{L},
\end{align}
\end{subequations}
where the parameter $\alpha_{ij}$ is a penalty scaling factor on each of the lines between nodes $i$ and $j$. For this problem, we set $\alpha_{ij} = \overline{s}_{ij}^2$ 

We model the time-varying demand as $d_{i,t+1} = d_{i,t} + \mathcal{N}\left(0, \frac{5}{t}\right)$.

The empirical average regret and cumulative violation are presented in Figure~\ref{reg300} for the IEEE 300-bus system~\cite{Matpower}. Similarly to Section~\ref{ssec:target}, multiple sketch percentages are used,

where $\tau = \lfloor \rho \card(\mathcal{L})\rfloor$ because the effective number of variables has been reduced by the equality constraints. For each sketch size, $1000$ trajectories are sampled over a time horizon of $T=500$ steps. 

The average regret and standard deviation for each sketch percentage are
 presented in Figure~\ref{fig:regvio} for \texttt{OSNR-EC}. Higher sketch percentages produce improved performance as more Hessian information is retained at every step. As with \texttt{the OSNR} method, a sketch percentage of 100\% reduces the update to the full-information Newton's method~\cite{OPENM}, which can be interpreted as a benchmark because no other OCO method is tailored to time-varying equality constraints. Performance improves as the sketch percentage increases, and even the lowest sketch percentage produces adequate results.
 Regarding the average constraint violation in Figure~\ref{vio300}, because the iterates are always feasible with respect to the last-known constraints, the constraint violation is the same for all sketch percentages, as expected. Finally, the execution times of \texttt{OSNR-EC}'s update as a function of sketch percentage is presented in Figure~\ref{fig:timesPower}. It namely illustrate the reduction in computing time in comparison to~\cite{OPENM} ($\rho=100\%$).

 \begin{figure}[htbp]
  \centering
\begin{subfigure}[t]{.9\textwidth}
\centering
\includegraphics[width=1\textwidth]{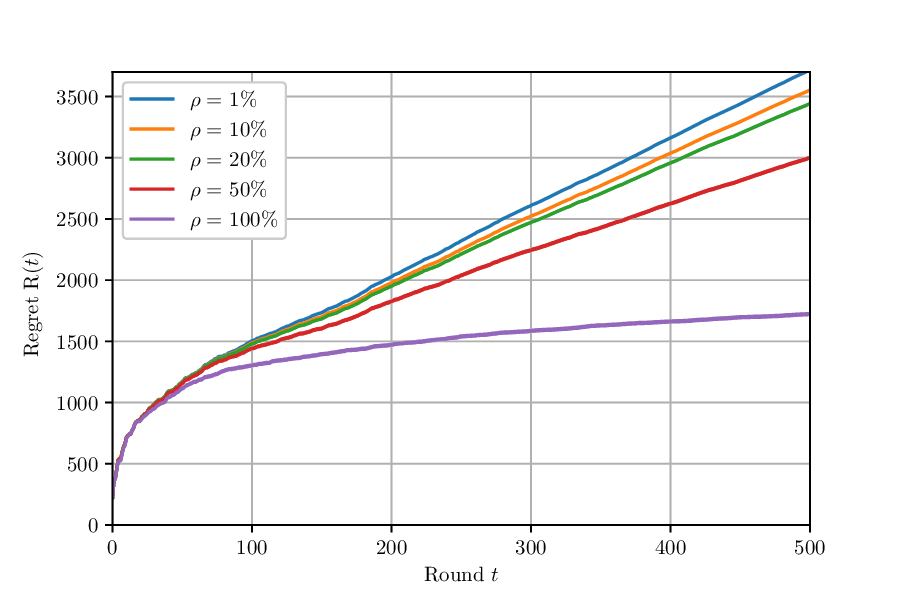} 
 \caption{Average regret.}
 \label{reg300}
 \end{subfigure}

\begin{subfigure}[t]{.9\textwidth}
\centering
\includegraphics[width=1\textwidth]{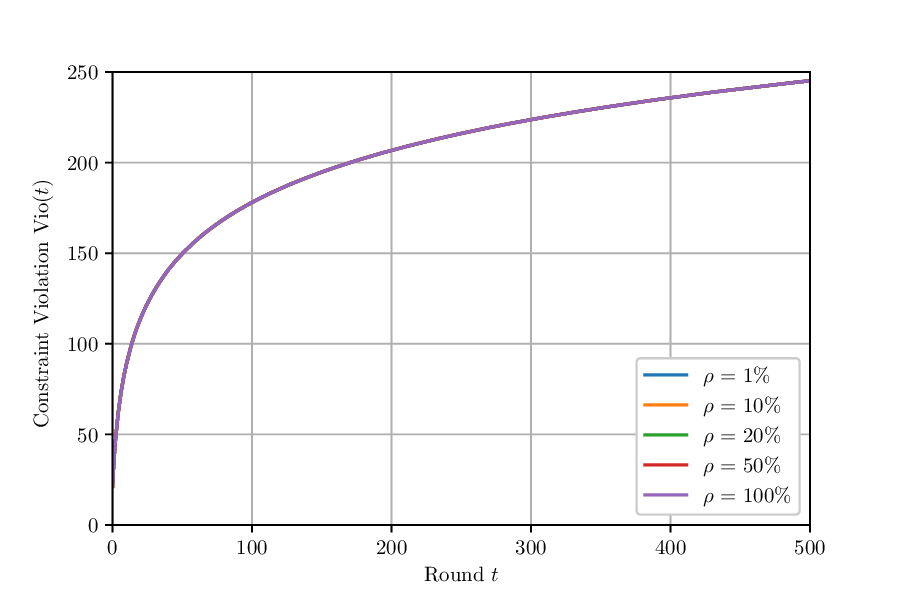} 
\caption{Constraint violation.}
\label{vio300}
\end{subfigure}
 \caption{Average regret and cumulative variation of \texttt{OSNR-EC} for the IEEE 300-bus system for different sketch percentages.}
 \label{fig:regvio}
 \end{figure}

 \begin{figure}[htbp]
  \centering
  \includegraphics[width=0.85\textwidth]{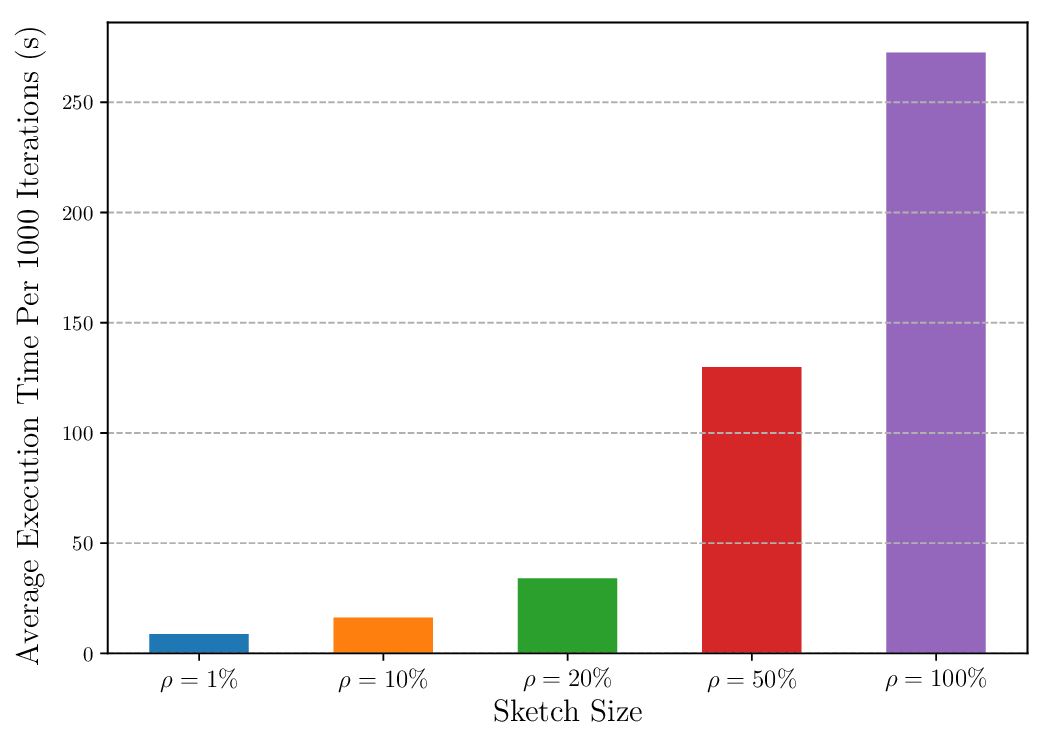} 
 \caption{Comparison of the \texttt{OSNR-EC}'s average execution times.}
 \label{fig:timesPower}
 \end{figure}

\section{Conclusion}\label{sec:conc}

In this paper, we present the online sketched Newton-Raphson method (\texttt{OSNR}), a reduced-computational-complexity, second-order method for online root-finding. Under the quasi-convexity assumption, we bound the dynamic regret of \texttt{OSNR} by $\mathcal{O}(V_T+1)$ where $V_T$ is the 
cumulative variation up to the time horizon $T$. This result is shown to extend to unconstrained and time-varying linear equality-constrained online convex optimization. For the latter, a projection step is added to form the \texttt{OSNR} method with equality constraints (\texttt{OSNR-EC}) leading to $\mathcal{O}(V_b)$ constraint violation bound. 

Two numerical applications are then considered: a high-dimensional point-tracking problem and a time-varying optimal power flow problem. When applied to the point-tracking problem, \texttt{OSNR} outperforms online gradient descent  for all sketch percentages. On the IEEE 300-bus system, \texttt{OSNR-EC} also generates high quality power dispatches even at very low sketch percentages. Both empirical regret and constraint violations are shown to be sublinear. 

In future work, we will consider inequality constraints within the online sketched Newton-Raphson framework. The objective is to expand the scope of the method to tackle more complex online optimization scenarios while controlling the computational cost of the method. Investigating the effect of different sketching processes is also a possible future direction.

\bibliographystyle{plain}
\bibliography{ref.bib}

\end{document}